\documentclass[oneside,english]{amsart}
\usepackage{amssymb}
\usepackage{amscd}
\usepackage{amsfonts, amsmath, amssymb, graphicx, epstopdf,amstext}
\usepackage{amsthm}
\usepackage{chngcntr}
\usepackage{etoolbox}
\usepackage{lipsum}
\numberwithin{equation}{section} 
\numberwithin{figure}{section} 
\theoremstyle{plain}
\newtheorem*{thm*}{Theorem}
\theoremstyle{plain}
\newtheorem{thm}{Theorem}[section]
\theoremstyle{definition}
\newtheorem{defn}[thm]{Definition}
\theoremstyle{plain}
\newtheorem{lem}[thm]{Lemma}
\theoremstyle{plain}
\newtheorem{prop}[thm]{Proposition}
\theoremstyle{plain}

\theoremstyle{remark}
\newtheorem{rem}[thm]{Remark}
\theoremstyle{remark}

\theoremstyle{remark}
\newtheorem*{acknowledgement*}{Acknowledgement}

\begin{document}

\title[NEW CLASSES OF PROJECTIVELY RELATED FINSLER METRICS OF CONSTANT FLAG CURVATURE]{NEW CLASSES OF PROJECTIVELY RELATED FINSLER METRICS OF CONSTANT FLAG CURVATURE}

\author[Cre\c tu]{Georgeta Cre\c tu}
\address{Faculty of  Mathematics \\ Alexandru Ioan Cuza University \\ Ia\c si, 
  Romania}
\email{cretuggeorgeta@gmail.com}
\date{\today}

\begin{abstract}
We define a  Weyl-type curvature tensor of $(1,2)$-type to provide a characterization for Finsler metrics of constant flag curvature. This Weyl-type curvature tensor is projective invariant only to projective factors that are Hamel functions. Based on this aspect we construct new families of projectively related Finsler metrics that have constant flag curvature. 
\end{abstract}


\maketitle

\section{Introduction}
One of the fundamental problems in Finsler geometry is to study and classify Finsler
metrics of constant (or scalar) curvature. Many Finslerian geometers have made effort to
study Finsler metrics of constant (or scalar) curvature \cite{Ba, Mo}.
Inspired by the Weyl-type curvature tensor of $(1,1)$-type  $W_0$  introduced in \cite{bc} and by the link between the curvature tensor associated to the nonlinear connection and the Jacobi endomorphsim, we define a new Weyl-type curvature tensor,  of $(1,2)$-type denoted $W_1$. Using the new tensor we provide a characterization for Finsler metrics of constant flag curvature. We treat separately the $2$-dimensional case in Theorem \ref{cfc-dim2} and the dimension greater than $3$ in Theorem \ref{thm:cfc1}. We study what happens with the projective Weyl-curvature tensor $W_1$ if we make a projective deformation of the initial spray. Similar to the tensor introduced in \cite{bc} we obtain that $W_1$ is invariant only to  projective factors that are Hamel functions. The main aim of this paper is to construct some families of projectively related Finsler metrics that preserve the Weyl-type curvature tensor $W_1$. \\
Due to the properties of projectively flat metrics, we choose, as a starting point for our study a family of projectively flat Randers metrics whose projective factor is proportional to the metric. Using the definition of a Randers metric we recover the result from \cite{CH} which states that a Randers metric $F=a+b$ is projectively flat if and only if $a$ is a projectively flat  Riemannian metric and $b$ is given by a closed $1$-form, \cite{Mo1}. 
In order to find the Randers metric with the property mentioned above we provide a characterization for projectively flat Finsler metrics that are reducible to a Riemannian metric, Theorem \ref{flat:riemm}. 
In the last part of this paper, inspired by \cite{ZSh}, we use different deformations of projectively flat Randers metrics in order to find new classes of projectively related Finsler metrics that preserve the property of having constant flag curvature.\\
The first family of metrics, \eqref{samecurvature}, was obtained through a Randers-type deformation of the projectively flat Randers metric \eqref{Finsler:m}. The most important aspect that deserves to be mentioned here is that the projective factor associated to the deformation mentioned above is linear and therefore is a Hamel function. Hence according to Proposition \ref{rem:curvature}, it follows that the metric obtained through this deformation is of constant flag curvature. Moreover, by computing the flag curvature we noticed that it is negative and equal to the one of the initial metric.
The second and the third family of metrics constructed in Subsection \ref{zero:curv} are both of zero flag curvature. The construction of the first family of metrics of zero flag curvature, \eqref{Berwald}, was inspired by the definition of square metrics, \cite{SH}, which plays a particular role in Finsler geometry, just as Randers metrics.
They are expressed by $F=\dfrac{(a+b)^2}{a}$, where $a$ and $b$ are the quantities from the definition of a Randers metric.  For the second construction we were inspired by the definition of conformal changes of Finsler spaces, which has been initiated by M. S. Knebelman, \cite{knebelman}.
The common point of the new metrics from Subsection \ref{zero:curv} is the fact that the projective factor obtained through the deformations mentioned above is proportional to the initial metric.
\section{Preliminaries}
\label{Intr}
We consider $M$ a connected, smooth, real and n-dimensional manifold. In this work, all geometric structures are smooth.
We denote by $C^\infty(M)$ the set of smooth functions on $M$, by $\mathfrak{X}(M)$ the set of vector fields on $M$ and by $\Lambda^k(M)$ the set of k-forms on $M$.
Local coordinates on $M$ are denoted by $(x^i)$, while induced local coordinates on the tangent bundle $TM$ are denoted by $(x^i, y^i)$ for $i=\bar{1,n}.$\\ We denote by $T_0M$, the tangent bundle with the zero section removed.
On $TM$ there are two canonical structures that we will use further, the Liouville vector field and the tangent endomorphism given in local coordinates by
\begin{equation}
\mathcal{C}=y^i\frac{\partial}{\partial y^i},\ J=dx^i\otimes\frac{\partial}{\partial y^i}.
\end{equation}
A system of second order ordinary differential equations on $M$,
\begin{equation}\label{geodesic-system}
\frac{d^2x^i}{dt^t}+2G^i\left(x,\frac{dx}{dt}\right)=0,
\end{equation}
can be identified with a special vector field on $TM$
\begin{equation}
S=y^i\frac{\partial}{\partial x^i}-2G^i(x,y)\frac{\partial}{\partial y^i},
\end{equation}
which satisfies $JS=\mathcal{C}$. This vector field is called a semispray. If additionally, $S\in \mathfrak{X}(T_0M)$ and satisfies $[\mathcal{C},S]=S$ we say that $S$ is a spray.\\
If we reparameterize the second-order system \eqref{geodesic-system}, by preserving the orientation of the
parameter, we obtain a new system and hence a new spray $\bar{S}=S-2P\mathcal{C}$.
The function $P\in C^\infty(T_0M)$ is 1-homogeneous, which means that it satisfies $\mathcal{C}(P)=P$ and it is related to the new parameter by
 \begin{equation}
\frac{d^2\bar{t}}{dt^2}=2P\left(x^i(t),\frac{dx^i}{dt}\right)\frac{d\bar{t}}{dt},\ \frac{d\bar{t}}{dt}>0. 
\end{equation}  
The two sprays S and $\bar{S}$ are called projectively related, while the function P is called
a projective deformation of the spray S. 
Every spray induces  a canonical nonlinear connection through the corresponding horizontal and vertical projectors, \cite{Grifone72a},
\begin{equation}
h=\frac{1}{2}(Id-[S,J]),\ v=\frac{1}{2}(Id+[S,J]).
\end{equation}
Locally, the two projectors $h$ and $v$ can be expressed as follows
\begin{equation}
h=\frac{\delta}{\delta x^i}\otimes dx^i,\ v=\frac{\partial}{\partial y^i}\otimes \delta y^i,
\end{equation} 
with
\begin{equation}
\frac{\delta}{\delta x^i}=\frac{\partial}{\partial x^i}-N_i^j(x,y)\frac{\partial}{\partial y^j},\ \delta y^i=dy^i+N_j^i(x,y)dx^j,\ N_j^i(x,y)=\frac{\partial G^i}{\partial y^j}(x,y).
\end{equation}
Alternatively, the nonlinear connection induced by a spray S can be characterized in
terms of an almost complex structure, \cite{Gr}
\begin{equation}
\mathbb{F}=h\circ[S,h]-J=\frac{\delta}{\delta x^i}\otimes \delta y^i-\frac{\partial}{\partial y^i}\otimes dx^i.
\end{equation}
For a spray $S$ consider the vector valued semi-basic 1-form
\begin{equation}
\Phi=v\circ[S,h]=R_j^i(x,y)\frac{\partial}{\partial y^i}\otimes dx^j,\ R_j^i=2\frac{\delta G^i}{\delta x^j}-S(N_j^i)+N_k^iN_j^k,
\end{equation}
which is be called the \emph{Jacobi endomorphism.}\\
Another important geometric structure induced by a spray $S$ is the curvature tensor $R$ associated to the nonlinear connection. It is the semi-basic 2 form
\begin{equation}\label{r}
R=-\frac{1}{2}[h,h].
\end{equation}
The curvature tensors $\Phi$  and $R$ are related by
\begin{equation}\label{relation:phi-R}
3R=[J,\Phi],\ \Phi=i_SR.
\end{equation} 
As we will see, important geometric information about the given spray
S are encoded in the Ricci scalar, $\rho\in C^\infty(T_0M)$ given by
\begin{equation}
\rho=\frac{1}{n-1}R_i^i=\frac{1}{n-1}Tr(\Phi).
\end{equation} 
\begin{defn}
A spray $S$ is said to be \textit{isotropic} if there exists a semi-basic 1-form $\alpha\in\Lambda^1(T_0M)$ such that the Jacobi endomorphism can be written as follows 
\begin{equation}\label{isotr}
\Phi=\rho J-\alpha\otimes \mathcal{C}.
\end{equation}
\end{defn}
In order to complete the geometric setting of a spray, we recall the Berwald connection. It is a linear connection on $T_0M$, for $X, Y \in \mathfrak{X}(T_0M),$ given by, \cite{BMiron}
\begin{equation}
D_XY=h[vX,hY]+v[hX,vY]+(\mathbb{F}+J)[hX,JY]+J[vX,(\mathbb{F}+J)Y].
\end{equation}
The action of the Berwald connection in the direction of the given spray $S$ provides
a tensor derivation on $T_0M$, which is called the dynamical covariant derivative define by
 $\nabla:\mathfrak{X}(T_0M)\rightarrow\mathfrak{X}(T_0M)$, given by
\begin{equation}
\nabla=h\circ\mathcal{L}_S\circ h+v\circ \mathcal{L}_S\circ v.
\end{equation}
For a spray $S$ and a function $L$ on $T_0M$, consider the following semi-basic 1-form, called the
Euler-Lagrange 1-form,
\begin{equation}\label{EL}
\delta_SL=\mathcal{L}_Sd_JL-dL=d_J\mathcal{L}_SL-2d_hL=\left\{S\left(\frac{\partial L}{\partial y^i}\right) - \frac{\partial L}{\partial x^i}\right\} dx^i.
\end{equation}
A $1$-homogeneous function that satisfies the Euler Lagrange equation $\delta_SL=0$ is called  \emph{Hamel function.}
\begin{defn}
By a \textit{Finsler function} we mean a continuous function $F : TM \rightarrow \mathbb{R}$
satisfying the following conditions:
\begin{enumerate}
\item $F$ is smooth and strictly positive on $T_0M$.
\item $F$ is positively homogeneous of order 1, which means that $F(x,\lambda y)=\lambda F(x,y)$, for all $\lambda> 0$ and $(x,y)\in TM.$
\item   The metric tensor with components
\[g_{ij}(x,y)=\frac{1}{2}\frac{\partial^2 F}{\partial y^i\partial y^j} \text{ has rank $n$ on $T_0M$}.\]
\end{enumerate}
\end{defn}
\begin{rem}
The regularity condition 3) of the previous definition is equivalent to the fact that the
Poincar\' e-Cartan 2-form of $F^2$, $\omega_{F^2} = -dd_JF^2$
, is non-degenerate and hence it is a
symplectic structure. Therefore, the equation
\begin{equation}\label{metrizable}
i_Sdd_JF^2=-dF^2
\end{equation}
uniquely determine a vector field $S$ on $T_0M$, which is called the \textit{geodesic spray}
of the Finsler function.
\end{rem}
\begin{defn}
A spray $S\in\mathfrak{X}(T_0M)$ is called \emph{ Finsler metrizable} if there is a Finsler function $F $ that satisfies \eqref{metrizable}. 
\end{defn}
We recall that in the Finslerian case we work with the notion of flag curvature. The notion of flag curvature extends to the Finslerian setting the concept of sectional curvature from the Riemannian
setting.
\begin{defn}
Consider $F$ a Finsler function and $\Phi$ the Jacobi endomorphism of its geodesic spray $S$. $F$ is said to be of\textit{ scalar (constant) flag curvature} if there exists a scalar function (constant) $\kappa$ on $T_0M$, such that
\begin{equation}
\Phi=\kappa F^2J-\kappa Fd_JF\otimes\mathcal{C}.
\end{equation}
Based on the relation between the Jacobi endomorphism and the curvature tensor $R$ written in \eqref{relation:phi-R} we get the following form for the curvature tensor associated to a Finsler function of scalar flag curvature
\begin{equation}\label{CFC-C}
R=\frac{1}{3F}d_J(\kappa F^3)\wedge J-d_J\bigg(\frac{1}{3F}d_J(\kappa F^3)\bigg)\otimes\mathcal{C}.
\end{equation}
\end{defn}
Moreover, if the Finsler function  is of constant flag curvature, then the curvature tensor has a simpler form given by
\begin{equation}\label{CFC}
R=\kappa Fd_JF\wedge J.
\end{equation}

\section{New Weyl-type curvature tensor} 
The main result of this section is a new characterization for Finsler metrics of constant flag curvature. This characterization  is based on a new Weyl-type curvature tensor inspired by the relation between the curvature tensor associated to the nonlinear connection and the Jacobi endomorphism. We will analyze first what happens in $dim M\geq 3$ in Theorem \ref{thm:cfc1} and we will treat separately the 2-dimensional case in Theorem \ref{cfc-dim2}.\\ Consider $S$ a geodesic spray with Jacobi endomorphism $\Phi$. We recall the Weyl-type curvature tensor $W_0$ introduced in \cite{bc}:
\begin{eqnarray}
W_0=\Phi - \frac{1}{n-1} \left(\operatorname{Tr} \Phi\right) J +
  \frac{1}{2(n-1)} d_J\left(\operatorname{Tr} \Phi\right) \otimes
  {\mathcal C}. \label{w0} \end{eqnarray}
Inspired by \eqref{relation:phi-R} and the techniques used in \cite[\S 8.3]{SLK14} we use the Weyl-type curvature tensor \eqref{w0} to introduce the following:
\begin{eqnarray*}
3W_1 &  =  & [J,W_0]=\left[J,\Phi-\frac{1}{n-1} \left(\operatorname{Tr} \Phi\right) J +
  					\frac{1}{2(n-1)} d_J\left(\operatorname{Tr} \Phi\right) \otimes
  					{\mathcal C}\right]\\ &  =  & \left[J,\Phi\right]-\left[J,\frac{1}{n-1} 											\left(\operatorname{Tr} \Phi\right) J\right]+\frac{1}{2(n-1)}
  					\left[J,d_J\left(\operatorname{Tr} \Phi\right) \otimes
  					{\mathcal C}\right]\\ & = & 
  					3R-\frac{1}{n-1}d_J\left(\operatorname{Tr} \Phi\right)\wedge J-\frac{1}{2(n-1)}d_J							\left(\operatorname{Tr} \Phi\right)\wedge\left[J,\mathcal{C}\right]\\& = &
  					3R-\frac{1}{n-1}d_J\left(\operatorname{Tr} \Phi\right)\wedge J-\frac{1}{2(n-1)}d_J							\left(\operatorname{Tr} \Phi\right)\wedge J\\&=&3R-\frac{3}{2(n-1)}d_J											\left(\operatorname{Tr} \Phi\right)\wedge J.
\end{eqnarray*} 
Consider $S$ a geodesic spray with Jacobi endomorphism associated $\Phi$ and curvature tensor $R$. We define the second  Weyl-type curvature tensor
\begin{equation}\label{w1}
W_1=R-\frac{1}{2(n-1)}d_J\left(\operatorname{Tr} \Phi\right)\wedge J.
\end{equation}

\begin{thm} \label{thm:cfc1}
A Finsler metric on a manifold of dimension greater than or equal to $3$ has constant flag curvature if and only if the Weyl-type tensor \eqref{w1} vanishes. 
\end{thm}
\begin{proof}
We start the direct implication with the assumption that $F$ has constant flag curvature. It follows that the geodesic spray $S$ is isotropic. Hence the curvature tensor can be written in the following form using \eqref{CFC}:
\begin{equation}\label{rtens}
R=\frac{1}{2}d_J(\kappa F^2)\wedge J=\frac{1}{2}d_J\rho\wedge J.
\end{equation}
Because the spray $S$ is Finsler metrizable by a function of constant flag curvature it follows that 
$\rho=\frac{Tr(\phi)}{n-1}$ and therefore the curvature tensor can be written in the following form
\begin{equation}\label{r1tens}
R=\frac{1}{2(n-1)}d_J\left(\operatorname{Tr}\Phi\right)\wedge J.
\end{equation}
Using the relation \eqref{r1tens} and the definition for the Weyl-type curvature tensor  $W_1$ we get that $W_1$ vanishes for a Finsler function of constant flag curvature.\\
For the converse, we assume that the Weyl-type curvature tensor \eqref{w1} vanishes, so the curvature tensor $R$ is given by formula \eqref{r1tens}.
Since $S$ is metrizable it follows that it satisfies the equation \eqref{metrizable} that is equivalent to   the Euler Lagrange equation $\delta_SF^2=0$. The last relation can be written as
\begin{equation}\label{d:f}
\delta_SF^2=d_JSF^2-2d_hF^2=0.
\end{equation}
Due to the fact that $F$ is the Finsler function that metricizes the geodesic spray $S$ from \eqref{d:f} we have
\begin{equation}\label{e-l}
d_hF^2=0.
\end{equation} 

From \eqref{e-l} and \eqref{r} we get 
\begin{equation}\label{der_r}
d_RF^2=0.
\end{equation}
Replacing the curvature tensor from \eqref{r1tens} in \eqref{der_r} we will obtain
\begin{equation}
\displaystyle d_{\frac{1}{2(n-1)}d_J\left(\operatorname{Tr}\Phi\right)\wedge J}F^2=0.
\end{equation}
The previous derivative can be rewritten in the following form
\begin{equation}\label{deriv}
\frac{1}{2(n-1)}d_J\left(\operatorname{Tr}\Phi\right)\wedge d_JF^2=0.
\end{equation}
The fact that $S$ is isotropic ensures that \eqref{deriv} implies
\begin{equation}
d_J\rho\wedge d_JF^2=0,
\end{equation}
which leads to 
\begin{equation}\label{djrho}
d_J\rho=\kappa d_JF^2, \text{ for some function } \kappa.
\end{equation} 
By applying $i_S$ in \eqref{djrho} we obtain that this function is given by:
\begin{equation}
\kappa=\frac{\rho}{F^2}=\frac{Tr(\Phi)}{(n-1)F^2}.
\end{equation}
From \eqref{djrho} we can notice that
\begin{equation}
d_J\kappa=0.
\end{equation}
 It follows that $R$ is given by \eqref{CFC-C}  $\kappa$  is of scalar flag curvature. Since $d_J\kappa=0$ it follows that $\kappa$ is constant along the fibers of the tangent bundle and according to the Finslerian version of the Schur Lemma, \cite{MM},  we get that $\kappa$ is constant and hence the Finsler metric has constant flag curvature.
 \end{proof}  
 The problem that appears in the 2-dimensional case is that we cannot apply the Finslerian version of Schur Lemma. Hence, we are forced to add an extra condition in order to obtain the conclusion of the previous theorem. Since $2$-dimensional sprays are isotropic we can use the semi-basic $1$-form $\alpha$ from \eqref{isotr} to formulate the following result:
\begin{thm}\label{cfc-dim2}
A Finsler metric on a $2$-dimensional manifold  has constant flag curvature if and only if the following conditions are satisfied
\begin{enumerate}
\item The Weyl-type tensor \eqref{w1} vanishes.
\item $\displaystyle d_h\alpha=0$.
 \end{enumerate}
\end{thm}
\begin{proof} 
For the direct implication we assume that $F$ is a Finsler function of constant flag curvature. Since we are in the $2$-dimensional case it follows that the spray is isotropic and the curvature tensor is given by \eqref{CFC}. The information that $F$ is of constant flag curvature leads to
\begin{equation}
R=\frac{1}{2(n-1)}d_J(Tr\Phi)\wedge J,
\end{equation}
which ensures that the first condition of the theorem is satisfied.\\
For the second condition, the metrizability of the geodesic spray $S$ implies
\begin{equation}
d_hF^2=0\Rightarrow d_hd_JF^2=0\Rightarrow d_h\left(Fd_JF\right)=0\overset{\kappa=ct}{\Longrightarrow}d_h(\kappa Fd_JF)=0\Rightarrow d_h\alpha=0.
\end{equation} 
The previous relation together with the fact that $S$ is isotropic ensures that the second condition from the theorem is also satisfied.\\
For the converse we start with the assumption that both conditions of the theorem are satisfied and based on this assumption we prove that $F$ has constant flag curvature.
  First, the vanishing of the Weyl-type curvature tensor $W_1$ implies that
   \begin{equation}\label{d:Jk=0}
d_J\kappa=0.
\end{equation}
From  the fact that the flag curvature does not depend on the fiber coordinate  we obtain that \[\kappa:=\kappa(x).\] In the same time, the second condition of the theorem gives us the information that  \[d_h\kappa=0,\] which leads to the conclusion that $\kappa$ is constant. 
\end{proof}
Since the aim of this paper is to find new classes of projectively related Finsler metrics of constant flag curvature we study what happens with the Weyl-type curvature tensor $W_1$ if we make a projective deformation of the initial spray $S\rightarrow \bar{S}=S-2P\mathcal{C}.$ Therefore we have the following result:
\begin{lem}\label{lem:pw1}
Consider $S$ and $\bar{S}=S-2P{\mathcal C}$ two projectively related
sprays. The corresponding Weyl-type curvature tensors  $W_1$ are related by
\begin{eqnarray}
\overline{W_1}=W_1+\frac{1}{2}\delta_SP \wedge J+d_Jd_hP\otimes {\mathcal C}. \label{pw1}
\end{eqnarray}
\end{lem}
\begin{proof}
Using the hypothesis that $\bar{S}$ and $S$ are two projectively related sprays, we have the following formula for the corresponding curvature tensors associated [2.2, \cite{BM12}].
\begin{eqnarray}
\overline{R} = R + d_Jd_hP\otimes\mathcal{C}+(Pd_JP-d_hP)\wedge J. \label{projr}
\end{eqnarray}
Moreover, for two projectively related sprays S and $\bar{S} = S-2P\mathcal{C}$, the corresponding Jacobi endomorphisms are related by the following formula 
\begin{equation}
\overline{\Phi}=\Phi+(P^2-SP)J-(Pd_JP+d_JSP-3d_hP)\otimes\mathcal{C}
\end{equation} 
Using \eqref{projr}  the Weyl-type curvature tensor $W_1$ can be written as follows:
\begin{eqnarray*}
\overline{W_1} &  =  &\overline{R}-\frac{1}{2(n-1)}d_J\left(\operatorname{Tr}\overline{\Phi}\right)								\wedge J\\ & = & R + d_Jd_hP\otimes\mathcal{C}+(Pd_JP-d_hP)\wedge J-\frac{1}								{2(n-1)}d_J(\left(\operatorname{Tr}\Phi\right)+(n-1)\left(P^2-SP)\right)\wedge J
						\\& = & R + d_Jd_hP\otimes\mathcal{C}+(Pd_JP-d_hP)\wedge J-\frac{1}											{2(n-1)}d_J\left(\operatorname{Tr}\Phi\right)\wedge J-\frac{1}{2}d_J\left(P^2-SP							\right)\wedge J\\ 
						& = &R-\frac{1}{2(n-1)}d_J\left(\operatorname{Tr}\Phi\right)\wedge J+ d_Jd_hP								\otimes\mathcal{C}+	Pd_JP\wedge J -d_hP\wedge J-Pd_JP\wedge J+\frac{1}{2}d_JSP								\wedge J
						\\& = & W_1+\frac{1}{2}\delta_SP \wedge J+d_Jd_hP\otimes {\mathcal C}.
\end{eqnarray*} and formula \eqref{pw1} is satisfied.
\end{proof} The previous lemma ensures that the Weyl-type curvature tensor $W_1$ is invariant  only to  projective factors of the deformation that satisfies $\delta_SP=0$ and hence is a Hamel function. As we already proved in \cite{bc} in the Riemannian case the projective factor $P$ is always a Hamel function due to its linearity. Therefore, if we make a projective deformation with a Hamel function of a Finsler metric of constant flag curvature the metric obtained through this deformation will have constant flag curvature as well. 
We want to make sure that the previous statement is true in the $2$-dimensional case. Hence, we analyze the second condition from Theorem \ref{cfc-dim2} under the hypothesis that $P$ is a Hammel function.
\begin{lem}\label{p:relatedFM}
Consider  $S $ and $\overline{S}=S-2P{\mathcal C}$ two projectively related isotropic sprays with the property that $P$ is a Hamel function. Then the derivatives with respect to the horizontal projector of the semi-basic $1$-forms $\alpha$ and $\overline{\alpha}$ are related by
\begin{equation}\label{p:r:dhalpha}
d_{\overline{h}}\overline\alpha=d_h\alpha-d_RP-Pd_J\alpha+\alpha \wedge d_JP.
\end{equation}
\end{lem}
\begin{proof}
From \cite{BM12} we recall the following relation between the horizontal projectors associated to a projective deformation $\overline{S}=S-2P\mathcal{C}:$
\begin{equation}
\overline{h}=h-PJ-d_JP\otimes\mathcal{C}.
\end{equation} 
Moreover, the semi-basic $1$-forms $\overline{\alpha}$ and $\alpha$ are related by
\begin{equation}
\overline{\alpha}=\alpha+Pd_JP+d_JSP-3d_hP.
\end{equation}
Based on the two relations written above we obtain the following
\begin{equation}
\begin{aligned}\label{dha}
d_{\overline{h}}\overline{\alpha}&=d_{h-PJ-d_JP\otimes\mathcal{C}}\overline{\alpha}=d_h\overline{\alpha}-Pd_J\overline{\alpha}-d_JP\wedge \mathcal{L}_\mathcal{C}\overline{\alpha}=d_h\overline{\alpha}-Pd_J\overline{\alpha}-d_JP\wedge \overline{\alpha}\\
&=d_h\alpha-2d_hP\wedge d_JP+d_JSP\wedge d_JP-3Pd_Jd_hP+d_hd_JSP-3d_RP-Pd_J\alpha+\alpha \wedge d_JP\\
&=d_h\alpha+(d_JSP-2d_hP)\wedge d_JP-3Pd_Jd_hP+d_hd_JSP-3d_RP-Pd_J\alpha+\alpha \wedge d_JP\\
&=d_h\alpha+\delta_SP\wedge d_JP -3Pd_Jd_hP+d_hd_JSP-3d_RP-Pd_J\alpha+\alpha \wedge d_JP.
\end{aligned}
\end{equation}
By adding the condition that $P$ is  a Hamel function \eqref{dha} becomes
\begin{equation}\label{dha1}
\begin{aligned}
d_{\overline{h}}\overline{\alpha}&=d_h\alpha+2d_hd_hP-3d_RP-Pd_J\alpha+\alpha \wedge d_JP\\&=d_h\alpha-d_RP-Pd_J\alpha+\alpha \wedge d_JP,
\end{aligned}
\end{equation}
which is the conclusion of our lemma.
\end{proof}
\begin{prop}\label{rem:curvature}
We consider $F$ and $\overline{F}$ two projectively related Finsler metrics. If the initial metric $F$ is of constant flag curvature and the projective factor is a Hamel function then $\overline{F}$ is of constant flag curvature.
\end{prop}
\begin{proof}
For $dim\geq 3$ the proposition is true in view of Lemma \ref{pw1} and Theorem \ref{thm:cfc1}. The same lemma ensures that the first item of Theorem \ref{cfc-dim2} is true.
In order to complete the proof, we should verify if the second item of Theorem \ref{cfc-dim2} is true for a Finsler metric projectively related to a metric  $F$ of constant flag curvature under the assumption that the projective factor is a Hamel function.\\ 
The hypothesis that $F$ is of constant flag curvature provides the informations that $d_h\alpha=0$,\ $d_J\alpha=0$ and the curvature tensor $R$ is given by
\begin{equation}\label{R}
R=\alpha\wedge J
\end{equation} 
Further, using \eqref{R} we obtain
\begin{equation}\label{d:RP}
d_RP=d_{\alpha\wedge J}P=\alpha\wedge d_JP
\end{equation}
Using \eqref{d:RP} and the conclusions obtained from the fact that $F$ is of constant flag curvature in Lemma \ref{p:relatedFM} we will obtain that $d_{\overline{h}}\overline{\alpha}=0$.
\\Therefore according to Theorem \ref{cfc-dim2} it follows that the proposition is true in dimension two as well.
\end{proof}
Further, starting from a Finsler metric of constant flag curvature we try to construct a projectively related Finsler metric of constant flag curvature by imposing the condition that the projective factor is a Hamel function.
\section{A characterization for projectively flat Riemannian metrics of constant curvature}
The main subject of this section will be related to a special class of  Finsler metrics, namely Randers metrics, which arise from many areas in mathematics, physics and biology, \cite{An, Shi}. They are expressed in the form $F(x,y)=a+b$, where $a(x,y) =\sqrt{g_{ij}(x)y^iy^j}$ is a Riemannian metric and $b(x,y) = b_i(x)y^i$. Based on this aspect we will notice that by restricting the class of Randers metrics to the class of projectively flat Randers metrics, the main information will be provided by the Riemannian metrics involved. The starting point for studying projectively flat Finsler metrics was
Hilbert's 4th problem, that asks to study and characterize projectively flat Finsler metrics (with constant flag curvature) on an open domain in $\mathbb{R}^n$.
We recall that a Finsler metric is said to be projectively flat if the geodesics are straight lines as point sets.\\
Therefore, the geodesic spray of a projectively flat Finsler metric can be written as \begin{equation*}
S=S_0-2P\mathcal{C}, \text{ where } S_0=y^i\dfrac{\partial}{\partial x^i}\in\mathfrak{X}(\mathcal{U}\times\mathbb{R}) \text{ is the flat spray and $\mathcal{U}\subset\mathbb{R}^n$ is open and convex}.
\end{equation*}
A characterization for projectively flat Finsler metric was given by G. Hamel and it states that: 
A metric is projectively flat if and onlt if it satisfies the Hamel equation:\begin{equation}\label{proj:flat}\delta_{S_0}F=0.\end{equation}
In this case the projective factor $P(x,y)$ is given by
\begin{equation}
P(x,y)=\frac{S_0F}{2F}.\label{H2}
\end{equation}
Hamel's  characterization allows us to formulate the following result for projectively flat Randers metrics:
\begin{prop}\label{prflatR}
A Randers metric $F=a+b$ is projectively flat if and only if the Riemannian metric $a$ is projectively flat and the $1$-form $b_idx^i$ is closed.
\end{prop}
\begin{proof}
In order to prove this result, we use the characterization for projectively flat Finsler metrics provided by Hamel. Let us rewrite \eqref{proj:flat} for the Randers metric $F=a+b$. Therefore, we have
\begin{equation}\label{prf}
\delta_{S_0}(a+b)=0\Leftrightarrow \delta_{S_0}a+\delta_{S_0}b=0.
\end{equation}
Since $b$ is linear in the fiber coordinate it follows that $\delta_{S_0}b=0$ if and only if $b_idx^i$ is closed. Therefore, from \eqref{prf}  we can conclude that $F$ is projectively flat if and only if $a$ is projectively flat and $b_idx^i$ is closed. 
\end{proof}
Proposition \ref{prflatR} indicates that the study of projectively flat Randers metrics should start with some characterizations concerning projectively flat Riemannian metrics.
Using the Levi Civita equations, \cite{mikes}, we will formulate a charaterization for projectively flat Finsler metrics that are reducible to Riemannian metrics as follows.
\begin{lem}\label{flat:riemm}
Let $F=\sqrt{g_{ij}(x)y^iy^j}$ be a positive definite Finsler metric on a domain $\mathcal{U}\subset\mathbb{R}^n$ that is reducible to a Riemannian metric. Then $F$ is projectively flat  if and only if the following relation is satisfied 
\begin{equation}\label{levi-civita}
g_{ij,l}=2\psi_lg_{ij}+\psi_ig_{jl}+\psi_jg_{il},\ P(x,y)=\psi_l(x)y^l.
\end{equation}
In this case, $P$ is the projective factor of $F$.
\end{lem}
\begin{proof}
The only thing we have to prove is the equivalence between Hamel equation \eqref{proj:flat} and Levi Civita equation \eqref{levi-civita}. For the direct implication we will assume that $F$ is projectively flat and hence $\delta_{S_0}F=0.$ According to Proposition 2.2 from \cite{bc} this is equivalent with
\begin{equation}\label{cond2}
\delta_{S_0}F^2=2Pd_JF^2.
\end{equation} 
The previous relation can be written in terms of the dynamical covariant derivative $\nabla_0$ in the following form
\begin{equation}\label{cov:der}
\nabla_0\left(d_JF^2\right)-d_{h_0}F^2=2Pd_JF^2.
\end{equation}
In local coordinates \eqref{cov:der} takes the form 
\begin{equation}\label{nabla}
\left(\nabla_0\left(2g_{ij}y^j\right)-\frac{\partial F^2}{\partial x^i}\right) dx^i=4Pg_{ij}y^jdx^i.
\end{equation}
Since $\nabla_0=\mathcal{D}_{S_0}$, where $\mathcal{D}_{S_0}$ is the Berwald connection of the flat spray $S_0$, from \eqref{nabla} we obtain
\begin{equation}\label{nabla1}
\mathcal{D}_{y^l\frac{\partial}{\partial x^l}}(2g_{ij}y^j)-g_{jl,i}y^jy^l=4Pg_{ij}y^j\Leftrightarrow 2g_{ij,l}y^ly^j-g_{jl,i}y^jy^l=4Pg_{ij}y^j.
\end{equation}
The previous relation leads us to the conclusion that for a Finsler function that is reducible to a Riemannian metric the projective factor $P$ is linear in fiber coordinates. We assume that $P(x,y)=\psi_i(x)y^i$. 
We want to obtain more information about the metric $g_{ij}$, which is why we take the derivative successively with respect to the fiber coordinate. The derivative with respect to $y^l$ leads us to
\begin{equation}\label{d:yk}
2g_{ij,l}y^j+2g_{il,j}y^j-2g_{jl,i}y^j=4\frac{\partial P}{\partial y^l}g_{ij}y^j+4Pg_{il}.
\end{equation} 
Further, the derivative with respect to $y^j$ gives us
\begin{equation}\label{d:yj}
2g_{ij,l}+2g_{il,j}-2g_{jl,i}=4\frac{\partial P}{\partial y^l}g_{ij}+4\frac{\partial P}{\partial y^j}g_{il}.
\end{equation}
We take a cyclic permutation on \eqref{d:yj}
\begin{equation}\label{permutation}
2g_{jl,i}+2g_{ij,l}-2g_{li,j}=4\frac{\partial P}{\partial y^i}g_{jl}+4\frac{\partial P}{\partial y^l}g_{ji}.
\end{equation}
Finally, by gathering \eqref{d:yj} and \eqref{permutation} we obtain
\begin{equation}
g_{ij,l}=2\frac{\partial P}{\partial y^l}g_{ij}+\frac{\partial P}{\partial y^j}g_{il}+\frac{\partial P}{\partial y^i}g_{jl},
\end{equation}
which are exactly the equations  \eqref{levi-civita} we wanted to reach.\\
For the converse, we start from \eqref{levi-civita}, we  multiply with $y^iy^j$ and sum over $i$ and $j$.
\begin{equation}\label{lch}
\frac{\partial g_{ij}y^iy^j}{\partial x^l}=2g_{ij}y^iy^j\psi_l+\psi_jy^jg_{il}y^i+\psi_iy^ig_{jl}y^j.
\end{equation}
From \eqref{lch} we get
\begin{equation}\label{lch1}
d_{h_0}F^2=2F^2d_JP+Pd_JF^2,
\end{equation}
which implies
\begin{equation}\label{lch2}
S_0(F^2)=4F^2P.
\end{equation}
Therefore $P=\dfrac{S_0F}{2F}.$
From \eqref{lch1} if follows $2d_{h_0}F=d_JS_0F$, which is equivalent to $\delta_{S_0}F=0$ and hence the proof is complete.
\end{proof}
At the moment we have all the tools needed to prove the direct implication of the classical statement of Beltrami theorem. 
\begin{thm}\label{Beltrami}
Let $F=\sqrt{g_{ij}(x)y^iy^j}$ be a projectively flat Finsler metric that is reducible to a Riemannian metric. Then $F$ is of constant flag curvature.
\end{thm}
\begin{proof}
Since $F$ is a projectively flat metric it follows that is projectively related to the Euclidean metric.
The Euclidian metric is a metric of zero constant flag curvature. Moreover, in the Riemannian case the projective factor is a Hamel function. Hence, according to Proposition \ref{rem:curvature} it follows that $F$ is of constant flag curvature.
\end{proof}
We recall that the family of projectively flat Finsler metrics that are reducible to a Riemannian metric  is given by, \cite{mikes}:
\begin{equation}\label{g}
F=\frac{\sqrt{|y|^2+\mu\left(|x|^2|y|^2-\langle x,y\rangle^2\right)}}{1+\mu|x|^2}.
\end{equation}
In order to simplify our search we treat further only projectively flat Randers metrics whose projective factor is proportional to the metric. Using the family \eqref{g} and the characterization for projectively flat Randers metrics we can formulate the following result
\begin{lem}\label{randers}
The family of projectively flat Randers metrics of negative constant flag curvature whose projective factor is proportional to the metric is given by
\begin{equation}\label{Finsler:m} F=a+b
\text{ where } a=\frac{\sqrt{|y|^2-4c^2\left(|x|^2|y|^2-\langle x,y\rangle^2\right)}}{1-4c^2|x|^2}\text{ and } b=\frac{2c\langle x,y\rangle}{1-4c^2|x|^2}.
\end{equation}
In this case, the constant $c$ represents the coefficient of proportionality between the projective factor and the metric.
\end{lem}
\begin{proof}
The family of all projectively flat Finsler metrics that are reducible to Riemannian metrics is given by \eqref{g}. In what follows we will see what conditions must be added for $b$ in order to obtain the desired result.
Hence, from Proposition \ref{prflatR} we have that \[a=\frac{\sqrt{|y|^2+\mu\left(|x|^2|y|^2-\langle x,y\rangle^2\right)}}{1+\mu|x|^2}.\]
Let us compute the projective factor for the Randers metric $F=a+b$ with $a$ written above.
\begin{equation}\label{p1}
\begin{aligned}
P_0=\frac{S_0F}{2F}=\frac{S_0(a+b)}{2(a+b)}=\frac{1}{2(a+b)}\left(-\frac{2\mu\langle x,y\rangle a}{1+\mu|x|^2}+S_0b\right).
\end{aligned}
\end{equation}
We want to find $b$ such that $P_0=cF.$ From \eqref{p1} it follows
\begin{equation}\label{p1:a}
-\frac{2\mu\langle x,y\rangle a}{1+\mu|x|^2}+S_0b=2c(a^2+2a b+b^2).
\end{equation}
Therefore, by separating the quadratic part from the non-quadratic part in \eqref{p1:a} we get
\begin{equation}\label{BETA}
b=-\frac{\mu\langle x, y\rangle}{2c(1+\mu|x|^2)}  \text{ and } S_0b=2c(a^2+b^2).
\end{equation}
Let us see what happens after replacing $b$ given above in the second relation from \eqref{BETA} 
\begin{equation}
\begin{aligned}
-\frac{\mu}{2c}a^2+2cb^2=2c(a^2+b^2).
\end{aligned}
\end{equation}
Hence, the projective factor is proportional to the metric if and only if the Riemannian metric is of negative constant curvature
\begin{equation}
\mu=-4c^2.
\end{equation}
The previous relation give us the information that we have to work only with projectively flat Riemanian metrics of negative curvature. Hence we choose $a$ of the following form
\begin{equation}\label{alpha}
a=\frac{\sqrt{|y|^2-4c^2\left(|x|^2|y|^2-\langle x,y\rangle^2\right)}}{1-4c^2|x|^2}.
\end{equation} and 
\begin{equation}
b=\frac{2c\langle x,y\rangle}{1-4c^2|x|^2}.
\end{equation}
Therefore the family we were looking for is the one  written in \eqref{Finsler:m}.
\end{proof}
Moreover, we can compute the flag curvature for the family of metrics written above as follows:
\begin{equation}
\begin{aligned}
\kappa F^2&=P_0^2-S_0P_0=c^2F^2-2c^2F^2=-c^2F^2\Rightarrow \kappa=-c^2.
\end{aligned}
\end{equation}
\section{Projectively related Finsler metrics that transform the projective Weyl-type curvature tensor $W_1$ into a invariant tensor }

All the constructions from this section are based on the family of metrics introduced  in Lemma \ref{randers}. Hence we make different projective deformation of a projectively flat Randers metric whose projective factor is proportional to the metric.
\subsection{A construction for a new class of projectively related Finsler metrics with the same negative flag curvature} We start with $F$ the projectively flat Randers metric of negative constant flag curvature given by \eqref{Finsler:m}. We make a Randers deformation of the metric \begin{eqnarray}F\rightarrow \overline{F}=F+\overline{b},\end{eqnarray} where $\overline{b}$ is given by $\overline{b}(x,y)=b_i(x)y^i.$  
We have to study first if $F$ and $\overline{F}$ defined above are projectively related.
In other words we have that
\begin{equation}
\delta_S\overline{F}=0\Leftrightarrow \delta_S\overline{b}=0.
\end{equation}
Let us compute the projective factor associated to the deformation mentioned above.
 \begin{equation}\label{projective:factor}
P=\frac{S(\overline{F})}{2\overline{F}}=\frac{S(F+\overline{b})}{2(F+\overline{b})}=\frac{S(\overline{b})}{2(F+\overline{b})}.
\end{equation}
Since the geodesic spray associated to $F$ is $S=S_0-2cF\mathcal{C}$, \eqref{projective:factor} becomes
\begin{equation}\label{pr:factor}
P=\dfrac{S_0\overline{b}-2cF\overline{b}}{2(F+\overline{b)}}.
\end{equation}
We assume that the projective factor is proportional with  $\overline{b}=\psi_iy^i$. With this assumption it follows that $\delta_SP=0$ and hence the projective Weyl-type tensor $W_1$ is invariant. 
Hence, let us write the assumption regarding the proportionality 
\begin{equation}\label{prop}
P=\nu\overline{b},\ \nu\in\mathbb{R}.
\end{equation}
From \eqref{pr:factor} and \eqref{prop} we get
\begin{equation}\label{beta:eq}
S_0\overline{b}-2cF\overline{b}=2\nu F\overline{b}+2\nu\overline{b}^2.
\end{equation}
The previous relation leads us to the conclusion that the quantity $\nu+c$ vanishes. Everything else is either quadratic or linear in the fiber coordinate, while $F$ cannot be linear.\\
Further we must find $\overline{b}$ that satisfies the assumption mentioned above.  The relation that provides us information regarding  $\overline{b}$ is the quadratic part from \eqref{beta:eq}, meaning:
\begin{equation}
S_0\overline{b}-2\nu\overline{b}^2=0.
\end{equation}
In local coordinates the previous relation has the following form
\begin{equation}\label{eq-beta}
y^iy^j\frac{\partial \psi_i}{\partial x^j}-2\nu y^iy^j\psi_i\psi_j=0.
\end{equation}
The condition of projectively related metrics ensures that $\psi_i$ is a gradient,\ $\psi_i=\dfrac{\partial \psi}{\partial x^i}$, hence \eqref{eq-beta} becomes
\begin{equation}\label{eq:1b}
\frac{\partial^2 \psi}{\partial x^i\partial x^j}-2\nu\frac{\partial \psi}{\partial x^i}\frac{\partial \psi}{\partial x^j}=0.
\end{equation}
Relation \eqref{eq:1b} can be written in a different form
\begin{equation}\label{eq:2b}
\frac{\partial}{\partial x^j}\left(e^{-2\nu \psi}\frac{\partial \psi}{\partial x^i}\right)=0.
\end{equation}
After integrating \eqref{eq:2b} we get
\begin{equation}\label{eq:3b}
e^{-2\nu b}\frac{\partial b}{\partial x^i}=e_i, \text{ where $e_i$ are constants on } \mathcal{U}\subset\mathbb{R}^n.
\end{equation}
From \eqref{eq:3b} it follows
\begin{equation}
-\frac{1}{2\nu}\frac{\partial }{\partial x^i}\left(e^{-2\nu \psi}\right)=e_i\Rightarrow \frac{\partial }{\partial x^i}\left(e^{-2\nu \psi}\right)=-2\nu e_i.
\end{equation}
Then 
\begin{equation}
e^{-2\nu \psi}=-2\nu \left(\langle e,x\rangle+f\right),
\end{equation}
and hence
\begin{equation}
\psi=-\frac{1}{2\nu }\ln\left(-2\nu \left(\langle e,x\rangle+f\right)\right).
\end{equation}
Finally we got that the 1-form $\overline{b}$ is given by
\begin{equation}\label{beta}
\overline{b}(x,y)=y^i\frac{\partial \psi}{\partial x^i}=\frac{\langle e,y\rangle}{4\nu^2\left(\langle e,x\rangle+f\right)},
\end{equation}
\text{ where $e_i$ and $f$ are constants such that $F+\overline{b}$ is positive}.
The form of the family of metrics constructed through this Randers deformation is
\begin{equation}\label{samecurvature}
\overline{F}=\frac{\sqrt{|y|^2-4\nu^2\left(|x|^2|y|^2-\langle x,y\rangle^2\right)}}{1-4\nu^2|x|^2}-\frac{2\nu\langle x,y\rangle}{1-4\nu^2|x|^2}+\frac{\langle e,y\rangle}{4\nu^2\left(\langle e,x\rangle+f\right)}.
\end{equation}
We can recover the flag curvature for the new metric as follows
\begin{equation}
\begin{aligned}
\overline{\kappa}\overline{F}^2=\kappa F^2+P^2-SP=-c^2F^2+\nu^2\overline{b}^2-\nu S(\overline{b})&=-\nu^2F^2+\nu^2\overline{b}^2-2\nu^2\overline{b}(F+\overline{b})\\&=-\nu^2\overline{F}\Rightarrow \overline{\kappa}=-\nu^2
\end{aligned}
\end{equation}
In conclusion we can say that  by a Randers change $\overline{F}=F+\overline{b}$, of a projectively flat Randers metric $F$ of constant flag curvature, whose projective factor is proportional to the metric, we can obtain a new family of projectively flat Finsler metrics with the same constant flag curvature as $F$ if we consider $\overline{b}$ from  \eqref{beta}.
If we analyze the metric \eqref{samecurvature} separately we notice that it is a projectively flat Finsler metric whose projective factor is $F-\bar{b}$.
Moreover, if we fix $\nu=-\frac{1}{2}$ and $f=1$  we will obtain that $F$ is the Funk metric, 
\begin{equation}\label{funk}
F=\frac{\sqrt{|y|^2-(|x|^2|y|^2-\langle x, y\rangle^2)}}{1-|x|^2}+\frac{\langle x, y \rangle}{1-|x|^2},\ y\in T_xB^n,
\end{equation} 
and $\overline{F}$ is the generalized Funk metric, \cite{ZSh},
\begin{equation}\label{genfunk}
F=\frac{\sqrt{|y|^2-(|x|^2|y|^2-\langle x, y\rangle^2)}}{1-|x|^2}+\frac{\langle x, y \rangle}{1-|x|^2}+\frac{\langle e,y\rangle}{1+\langle e,x\rangle},\ \text{ with  $e$ constant vector }, e\in\mathbb{R}^n,\ |e|<1.
\end{equation} 
\subsection{New Finsler metrics of zero flag curvature}\label{zero:curv}
\subsubsection{A new class of projectively flat Finsler metrics of zero flag curvature}
 Starting from the projectively flat Finsler metric \eqref{Finsler:m} we construct a family of  Finsler metric of zero flag curvature having as starting point the definition of the family of square Finsler metrics, \cite{SH}.\\We consider the following deformation of the Finsler metric \eqref{Finsler:m}: \begin{equation}\label{w:f}
\overline{F}=f(x)\frac{F^2}{a}, \text{ where f is a positive function and } a \text{ is given by \eqref{alpha}}.
\end{equation}
The first step is to see under what condition $F$ and $\overline{F}$ are projectively related. Therefore, we formulate the following result
\begin{lem}\label{P:Randers}
Let $F=a+b$ a Randers metric. Then $\overline{F}=f(x)\dfrac{F^2}{a}$ is projectively related to $F$ if and only if the following relation is satisfied
\begin{equation}\label{Proj:related}
\frac{F^2}{a}d_hf-S(f)d_J\left(\frac{F^2}{a}\right)+fS(a)d_J\left(\frac{F^2}{a^2}\right)=0.
\end{equation}
\end{lem}
\begin{proof}
From \cite{bc} we have that $F$ and $\overline{F}$ are projectively related if and only if $\delta _S\overline{F}=0$. Hence, we must rewrite the condition $\delta_S\overline{F}=0$, or equivalently $d_Jd_h\overline{F}=0$ for the function $\overline{F}$ \eqref{w:f}.\\
Since $S$ is the geodesic spray that metricizes the Finsler function $F$ it follows that $d_hF=0.$ Moreover, since $F=a+b$, with $b$ linear it follows that $d_hd_Ja=0.$
The assumption regarding the function $f$ presented at the beginning of this subsection can be written as
\begin{equation}\label{finf}
d_Jf=0.
\end{equation}
Using \eqref{finf} it follows that 
\begin{equation}
d_Jd_h\overline{F}=0\Leftrightarrow d_J\left(\frac{F^2}{a}\right)\wedge d_hf-fd_J\left(\frac{F^2}{a^2}\right)\wedge d_ha=0.
\end{equation}
Moreover, the previous relation leads to
\begin{equation}
i_S\left(d_J\left(\frac{F^2}{a}\right)\wedge d_hf-fd_J\left(\frac{F^2}{a^2}\right)\wedge d_ha\right)=0,
\end{equation}
which due to the fact that $F$ and $a$ are $1$-homogeneous becomes \eqref{Proj:related}.
\end{proof}
It is difficult to find the function $f$ directly from the condition written in Lemma \ref{P:Randers}. This is why we assume that $F$ and $\overline{F}$ are projectively related and that the projective factor is a Hammel function, meaning $\delta_SP=0.$ Using these assumptions we try to find the function $f$ that satisfies the conditions mentioned in Lemma \ref{P:Randers}.
\begin{equation}\label{P:factor}
\begin{aligned}
\displaystyle P&=\frac{S\overline{F}}{2\overline{F}}=\dfrac{S(f)\dfrac{F^2}{a}-\dfrac{fF^2S(a)}{a^2}}{\dfrac{2fF^2}{a}}=\dfrac{S(f)\cdot\dfrac{F^2}{a}-\dfrac{fF^2}{a^2}\cdot\left(S_0a-2cFa\right)}{\dfrac{2fF^2}{a}}\\
&=\dfrac{S(f)}{2f}-\dfrac{S_0a-2cFa}{2a}=\dfrac{Sf}{2f}-\frac{4ca b-2cFa}{2a}=\dfrac{S_0f}{2f}-2cb+cF.
\end{aligned}
\end{equation} 
We can notice that \begin{equation}
\delta_SP=\delta_S\left(\frac{S_0f}{2f}-2cb+cF\right)=\delta_S\left(\frac{S_0f}{2f}\right)
\end{equation}
We recall that we want to find $f$ such that $\delta_SP=0$. Since $d_hf=df$  we can assume that \[S_0f=4cfb.\]
Based on the previous assumption we go back to formula \eqref{Proj:related} to see if $F$ and $\overline{F}$ are projectively related. Hence we have to see if the following relation is true
\begin{equation}\label{verific}
\frac{F^2}{a}d_hf-f\left(\frac{4cb d_JF^2}{a}-\frac{4cb F^2 d_Ja}{a^2}-\frac{S(a)d_JF^2}{a^2}+\frac{S(a)F^2d_Ja^2}{a^4}\right)=0.
\end{equation}
After replacing $S(a)=2ca(2b-F)$ in \eqref{verific} we get
\begin{equation}
\begin{aligned}
&\frac{F^2}{a}d_hf-f\left(\frac{4cb d_JF^2}{a}-\frac{4cb F^2 d_Ja}{a^2}-\frac{(2ca(2b-F))d_JF^2}{a^2}+\frac{(2ca(2b-F))F^2d_Ja^2}{a^4}\right)\\=&
\frac{F^2}{a}d_hf-f\left(\frac{4cb d_JF^2}{a}-\frac{4cb F^2 d_Ja}{a^2}-\frac{(2c(2b-F))d_JF^2}{a}+\frac{(4c(2b-F))F^2d_Ja}{a^2}\right)\\=&
\frac{F^2}{a}d_hf-f\left(\frac{4cF^2d_JF}{a}+\frac{4cb F^2d_Ja}{a^2}-\frac{4cF^3d_Ja}{a^2}\right).
\end{aligned}
\end{equation}
Therefore, in this case $F$ and $\overline{F}$ are projectively related if and only if
\begin{equation}
d_hf-4cf\left(d_JF+\frac{b d_Ja-Fd_Ja}{a}\right)=0.
\end{equation}
The assumptions made for $f$ ensures that
\begin{equation}
d_JS_0f-4cfd_Jb=0,
\end{equation}
which leads to the conclusion that the two metrics are projectively related.\\
Moreover we can compute the curvature for the new metric
\begin{equation}
\begin{aligned}
\overline{\kappa}\overline{F}^2&=\kappa F^2+P^2-SP\\&=-c^2F^2+c^2F^2=0\Rightarrow \overline{\kappa}=0.
\end{aligned}
\end{equation}
The last aspect we want to discuss here is the form of the function $f$. We  show that we can find $f$ that satisfies $S_0f=4cfb$, for $b$ given by \eqref{BETA}. We have, \[\frac{\partial f}{\partial x^i}=\frac{8c^2fx_i}{1-4c|x|^2}.\]
Therefore
\[\frac{\partial }{\partial x^i}\ln f=\frac{\partial}{\partial x^i}\ln\left(\frac{1}{1-4c^2|x|^2}\right)\Rightarrow f=\frac{\eta}{1-4c^2|x|^2}, \eta\in\mathbb{R^+}.\]
We can write now the form for the new metric
\begin{equation}\label{Berwald}
\overline{F}=\frac{\eta(\sqrt{|y|^2-4c^2\left(|x|^2|y|^2-\langle x,y\rangle^2\right)}+2c\langle x,y\rangle)^2}{(1-4c^2|x|^2)^2\sqrt{|y|^2-4c^2\left(|x|^2|y|^2-\langle x,y\rangle^2\right)}},\ \eta\in\mathbb{R^+}.
\end{equation}
The metric \eqref{Berwald} is a projectively flat Finsler metric whose projective factor is proportional to the projectively flat metric \eqref{Finsler:m}. Therefore the projective factor is a Hamel function and in view of Proposition \ref{rem:curvature} it follows that the metric \eqref{Berwald} is of constant flag curvature.\\
In this case, by fixing $c=\dfrac{1}{2}$ and $\eta=1$ we will recover the projective equivalence between the Funk metric \eqref{funk} and the Berwald metric, \cite{ZSh}:
\begin{equation}\label{Berw}
\overline{F}=\frac{\left(\sqrt{|y|^2-(|x|^2|y|^2-\langle x, y\rangle^2)}+\langle x, y \rangle\right)^2}{\left(1-|x|^2\right)^2\sqrt{|y|^2-(|x|^2|y|^2-\langle x, y\rangle^2)}}.
\end{equation}
\subsubsection{Another class of Finsler metrics of zero flag curvature obtained by conformal changes of Finsler spaces}
In this section we use a deformation of the family of  projectively flat Randers metrics \eqref{Berwald} to obtain a new class of Finsler metrics of zero flag curvature. Therefore, we will multiply the family of Finsler metrics of zero flag curvature with a $0$-homogeneous function in order to obtain a new class of projectively flat Finsler metrics of zero flag curvature.
We consider the following deformation of the projectively flat metric $\overline{F}$ given by \eqref{Berwald}:
\begin{equation}\label{conf}
\widetilde{F}=g\overline{F}+\dfrac{f}{F}\overline{F}, 
\end{equation} 
where $g$ is a function on $M$, $f(x,y)=f_i(x)y^i$ is a function in $TM$ such that $\widetilde{F}$ is positive and $F$ is the metric given in \eqref{Finsler:m}.
In this case we are interested under what conditions $\bar{F}$ is a projectively flat Finsler metric.
Hence we have the following result:
\begin{lem}\label{kzero}
We consider $F$ and $\overline{F}$ two projectively flat Finsler metrics and $\widetilde{F}=g\overline{F}+\dfrac{f}{F}\overline{F}$ a metric obtained by a multiplication of the projectively flat metric $\overline{F}$ with the $0$-homogeneous function $g+\dfrac{f}{F}$, where $g$ and $f$ are described in \eqref{conf} . Then $\widetilde{F}$ is projectively flat if and only if the following relation is satisfied
\begin{equation}\label{proj:flat:multiplication}
\begin{aligned}
&\overline{F}d_{h_0}g-S_0gd_J\overline{F}-\dfrac{S_0fd_J\overline{F}}{F}+\dfrac{\overline{F}S_0fd_JF}{F^2}-\dfrac{S_0\overline{F}d_Jf}{F}+\dfrac{fS_0\overline{F}d_JF}{F^2}+\dfrac{\overline{F}S_0Fd_Jf}{F^2}\\&+\dfrac{fS_0Fd_J\overline{F}}{F^2}-\dfrac{
2f\overline{F}S_0Fd_JF}{F^3}=0.
\end{aligned}
\end{equation} 
\end{lem}
\begin{proof}
We use the same motivation as in the proof of the Lemma \ref{P:Randers}. We compute first the derivative with respect to the horizontal projector associated to the flat spray, which is $h_0$
\begin{equation}
d_{h_0}\widetilde{F}=\overline{F}d_{h_0}g+gd_{h_0}\overline{F}+\dfrac{\overline{F}}{F}d_{h_0}f+\dfrac{f}{F}d_{h_0}\overline{F}-\dfrac{f\overline{F}}{F^2}d_{h_0}F.
\end{equation}
Further
\begin{equation}\label{djdhzero}
\begin{aligned}
d_Jd_{h_0}\widetilde{F}&=d_J\overline{F}\wedge d_{h_0}g+\overline{F}d_Jd_{h_0}g+d_Jg\wedge d_{h_0}\overline{F}+gd_Jd_{h_0}\overline{F}+\dfrac{d_J\overline{F}\wedge d_{h_0}f}{F}+\dfrac{\overline{F}d_Jd_{h_0}f}{F}\\&+\dfrac{\overline{F}d_{h_0}f\wedge d_JF}{F^2}+\dfrac{d_Jf\wedge d_{h_0}\overline{F}}{F}+\dfrac{fd_Jd_{h_0}\overline{F}}{F}+\dfrac{fd_{h_0}\overline{F}\wedge d_JF}{F^2}-\dfrac{\overline{F}d_Jf\wedge d_{h_0}F}{F^2}\\&-\dfrac{fd_J\overline{F}\wedge d_{h_0}F}{F^2}-\dfrac{2f\overline{F} d_{h_0}F\wedge d_JF}{F^3}.
\end{aligned}
\end{equation}
Since $F$ and $\overline{F}$ are projectively flat Finsler metrics it follows that $d_Jd_{h_0}F=d_Jd_{h_0}\overline{F}=0$. Moreover, $g$ is a function that does not depend on the fiber coordinate, hence $d_Jg=0.$ The linearity of the function $f$ ensures that $d_Jd_{h_0}f=0$ and hence from \eqref{djdhzero} it remains:
\begin{equation}
\begin{aligned}
d_Jd_{h_0}\widetilde{F}=&d_J\overline{F}\wedge d_{h_0}g+\dfrac{d_J\overline{F}\wedge d_{h_0}f}{F}+\dfrac{\overline{F}d_{h_0}f\wedge d_JF}{F^2}+\dfrac{d_Jf\wedge d_{h_0}\overline{F}}{F}+\dfrac{fd_{h_0}\overline{F}\wedge d_JF}{F^2}\\&-\dfrac{\overline{F}d_Jf\wedge d_{h_0}F}{F^2}-\dfrac{fd_J\overline{F}\wedge d_{h_0}F}{F^2}-\dfrac{2f\overline{F} d_{h_0}F\wedge d_JF}{F^3}.
\end{aligned}
\end{equation} 
In order to obtain an equivalent condition with \eqref{proj:flat} we apply $i_{S_0}$ in the previous relation and we obtain:
\begin{equation}
\begin{aligned}
\delta_{S_0}\widetilde{F}&=\overline{F}d_{h_0}g-S_0gd_J\overline{F}+\dfrac{\overline{F}}{F}d_{h_0}f-\dfrac{S_0fd_J\overline{F}}{F}+\dfrac{\overline{F}S_0fd_JF}{F^2}-\dfrac{\overline{F}d_{h_0}f}{F}+\dfrac{fd_{h_0}\overline{F}}{F}\\&-\dfrac{S_0\overline{F}d_Jf}{F}+\dfrac{fS_0\overline{F}d_JF}{F^2}-\dfrac{fd_{h_0}\overline{F}}{F}-\dfrac{f\overline{F}d_{h_0}F}{F^2}+\dfrac{\overline{F}S_0Fd_Jf}{F^2}-\dfrac{f\overline{F}d_{h_0}F}{F^2}\\&+\dfrac{fS_0Fd_J\overline{F}}{F^2}-\dfrac{
2f\overline{F}S_0Fd_JF}{F^3}+\dfrac{2f\overline{F}d_{h_0}F}{F^2}.
\end{aligned}
\end{equation}
In conclusion we have that $\widetilde{F}$ is a projectively flat Finsler metric if and only if
\begin{equation}
\begin{aligned}
&\overline{F}d_{h_0}g-S_0gd_J\overline{F}-\dfrac{S_0fd_J\overline{F}}{F}+\dfrac{\overline{F}S_0fd_JF}{F^2}-\dfrac{S_0\overline{F}d_Jf}{F}+\dfrac{fS_0\overline{F}d_JF}{F^2}+\dfrac{\overline{F}S_0Fd_Jf}{F^2}\\&-\dfrac{fS_0Fd_J\overline{F}}{F^2}-\dfrac{
2f\overline{F}S_0Fd_JF}{F^3}=0,
\end{aligned}
\end{equation}
which completes the proof.
\end{proof}
In this case we have two unknown quantities which makes it hard to find the new  projectively flat Finsler metric. This is why we compute first the projective factor associated to the deformation 
\begin{equation}\label{prfactorconform}
\begin{aligned}
P&=\dfrac{S_0\widetilde{F}}{2\widetilde{F}}=\dfrac{S_0g\cdot \overline{F}+gS_0\overline{F}+\dfrac{S_0f}{F}\overline{F}+\dfrac{fS_0\overline{F}}{F}-f\dfrac{\overline{F}}{F^2}S_0F}{2\left(g\overline{F}+\dfrac{f}{F}\overline{F}\right)}
\end{aligned}
\end{equation}
Since $F$ and $\overline{F}$ are two projectively flat Finsler metrics for which $S_0F=2cF^2$ and $S_0\overline{F}=4cF\overline{F}$ from \eqref{prfactorconform} we have:
\begin{equation}
\begin{aligned}
P&=\dfrac{S_0g\cdot \overline{F}+\dfrac{S_0f}{F}\overline{F}-2cf\overline{F}}{2\left(g\overline{F}+\dfrac{f}{F}\overline{F}\right)}+2cF.
\end{aligned}
\end{equation}
We recall that we want to find $g$ and $f$ such that the projective factor $P$ is a Hamel function. Taking into account the conditions imposed on the functions $f$ and $g$ it follows that we can make the following extra assumption 
\begin{equation}\label{assumconform}
S_0g=2cf \text{ and } S_0f=0.
\end{equation} 
With the assumptions considered in \eqref{assumconform} we get that the projective factor associated is 
\begin{equation}
P=2cF,
\end{equation}
which is a Hamel function.\\
Let us verify first that the two equations written in \eqref{assumconform} ensure the projective flatness of the Finsler metric $\widetilde{F}$. Hence we have to check if \eqref{proj:flat:multiplication} is satisfied under this hypothesis. First, let us notice that the only function that satisfies the second equation from \eqref{assumconform} is \begin{equation}
f=\langle v, y\rangle.
\end{equation}
Therefore, $S_0f=y^i\dfrac{\partial f}{\partial x^i}=0$. We still have to compute the following expression
\begin{equation}
\begin{aligned}
&\overline{F}d_{h_0}g-S_0gd_J\overline{F}-\dfrac{S_0\overline{F}d_Jf}{F}+\dfrac{fS_0\overline{F}d_JF}{F^2}+\dfrac{\overline{F}S_0Fd_Jf}{F^2}+\dfrac{fS_0Fd_J\overline{F}}{F^2}-\dfrac{
2f\overline{F}S_0Fd_JF}{F^3}\\&=\overline{F}d_JS_0g-2cfd_J\overline{F}-4cf\overline{F}d_Jf+\dfrac{4cf\overline{F}d_JF}{F}+2c\overline{F}d_Jf+2cfd_J\overline{F}-\dfrac{4cf\overline{F}d_JF}{F}\\&=\overline{F}\left( d_JS_0g-2cd_Jf\right)=0.
\end{aligned}
\end{equation}
The previous expression ensures that the condition from Lemma  \ref{kzero} is satisfied and hence $\widetilde{F}$ is a projectively flat Finsler metric. In order to find the form of $\bar{F}$ we have to find the form of the function $g$ from \eqref{assumconform}. Since
\begin{equation}
y^i\dfrac{\partial g}{\partial x^i}=2cv_iy^i
\end{equation}
it follows that $g(x)=2c\langle v,x\rangle+e.$\\
We can write now the expression for the new metric as follows
\begin{equation}\label{cfczero}
\small
\widetilde{F}=\dfrac{\eta(\sqrt{|y|^2-4c^2\left(|x|^2|y|^2-\langle x,y\rangle^2\right)}+2c\langle x,y\rangle)^2}{(1-4c^2|x|^2)^2\sqrt{|y|^2-4c^2\left(|x|^2|y|^2-\langle x,y\rangle^2\right)}}\cdot\left(2c\langle v,x\rangle+e+\dfrac{\langle v,y\rangle}{\frac{\sqrt{|y|^2-4c^2\left(|x|^2|y|^2-\langle x,y\rangle^2\right)}}{1-4c^2|x|^2}+\frac{2c\langle x,y\rangle}{1-4c^2|x|^2}}\right).
\end{equation}
We notice that the metric \eqref{cfczero} is a projectively flat Finsler metric with the same properties as the one previously constructed. Therefore, the projective factor is proportional to the metric \eqref{Finsler:m} and it is of zero flag curvature. The curvature can be recovered from the fact that the metric previously constructed is projectively related to the flat metric and hence
\begin{equation}
\widetilde{\kappa}\widetilde{F}^2=P^2-S_0P=4c^2F^2-2cS_0F=4c^2F^2-4c^2F^2=0\Rightarrow \widetilde{\kappa}=0.
\end{equation}
In conclusion, the constructions presented in this section are based on the obstruction imposed in Proposition \ref{rem:curvature} on the projective factor. For the construction of the first metric we used the fact that a linear function given by a closed $1$-form is always a Hamel function. Regarding the constructions from Subsection \ref{zero:curv} we can say they were based on the fact that a Finsler function $F$ is always a Hamel function for its geodesic spray or we used the fact that a projectively flat Finsler metric satisfies $\delta_{S_0}F=0$ and hence is a Hamel function for the flat spray. 

\end{document}